\documentclass[11pt]{amsart}
\usepackage[english]{babel}
\usepackage[T1]{fontenc}
\usepackage[latin1]{inputenc}
\usepackage{graphicx}
\usepackage{amsmath,amssymb,amsthm,mathrsfs,txfonts,ifthen,xr}
\usepackage{soul}
\usepackage{array, multirow}
\usepackage{stmaryrd}
\usepackage{color}
\usepackage{hyperref}
\usepackage{enumerate}
\usepackage[all]{xy}
\usepackage{tikz-cd}
\usetikzlibrary{decorations.pathreplacing}

\usepackage{ytableau}

\theoremstyle{plain}
\newtheorem{theo}{Theorem}[section]
\newtheorem*{theo*}{Theorem}

\newtheorem{proposition}[theo]{Proposition}

\theoremstyle{definition}
\newtheorem{definition}[theo]{Definition}

\newtheorem{notation}[theo]{Notation}

\theoremstyle{remark}
\newtheorem{remark}[theo]{Remark}

\def\A{{\rm A}}
\def\B{{\rm B}}
\def\C{{\rm C}}
\def\D{{\rm D}}

\def\F{{\rm F}}

\def\K{{\rm K}}

\def\O{{\rm O}}
\def\P{{\rm P}}
\def\p{{\rm p}}
\def\Q{{\rm Q}}
\def\q{{\rm q}}

\def\S{{\rm S}}

\def\T{{\rm T}}
\def\U{{\rm U}}
\def\V{{\rm V}}

\def\X{{\rm X}}
\def\Y{{\rm Y}}

\def\tr{{\rm tr}}

\def\Sp{{\rm Sp}}
\def\GL{{\rm GL}}	
\def\Id{{\rm Id}}

\def\Mat{{\rm Mat}}

\def\log{{\rm log}}

\def\Spec{{\rm Spec}}
\def\diag{{\rm diag}}

\allowdisplaybreaks
\title{Spectral Decomposition and Linearization of Kubo-Ando Means}
\author{Raluca Dumitru}
\address{Department of Mathematics and Statistics\\ University of North Florida \\ 1 UNF Drive \\ Jacksonville \\ FL 32224 \\ USA}
\email{raluca.dumitru@unf.edu}
\author{Jose Franco}
\address{Department of Mathematics and Statistics\\ University of North Florida \\ 1 UNF Drive \\ Jacksonville \\ FL 32224 \\ USA}
\email{jose.franco@unf.edu}
\author{Allan Merino}
\address{Department of Mathematics and Statistics\\ University of North Florida \\ 1 UNF Drive \\ Jacksonville \\ FL 32224 \\ USA}
\email{allan.merino@unf.edu}



\date{}

\begin{document}

\begin{abstract}

In this paper, we study the structure of Kubo-Ando means on the cone of positive Hermitian matrices over the real numbers, complex numbers, and quaternions. Given a Kubo-Ando mean $\sigma$ with representing function $f$, we obtain an explicit decomposition of $\A \sigma \B$ in terms of the spectrum of $\A^{-1}\B$. More precisely, we show that $\A \sigma \B$ can be expressed as a finite linear combination of matrices of the form $\A\left(\A^{-1}\B\right)^{k}$, with coefficients depending only on $f$ and the eigenvalues of $\A^{-1}\B$. We first investigate the linear case and characterize the pairs of matrices for which every Kubo-Ando mean admits an affine representation. We then focus on the cone $\mathscr{P}_{3}(\mathbb{D})$, where we derive explicit formulas for the decomposition coefficients in terms of spectral invariants. Finally, we show that the same techniques extend to a broad class of alternative means, yielding explicit decompositions in the commutative setting and extending recent results of Choi, Kim, and Lim.

\end{abstract}

\maketitle

\tableofcontents

\section{Introduction}

In the study of operator means, the theory by Kubo and Ando \cite{KUBOANDO} sits at the center of it. In particular, the affine isomorphism between the set of operator means and the set of standardized operator  monotone functions offers an invaluable tool for the study of means. However, in order to compute these means, it is necessary to compute square roots, inverses, and the operator monotone functions themselves\,.

\noindent This motivates the question: Is it possible to find a simpler expression for an arbitrary Kubo-Ando mean? Several such results exist in the literature. For example, in the case when $\A\,, \B \in\mathrm{SL}(2,\mathbb{C})$ are positive definite, 
\begin{equation*}
\A\sharp \B = \frac{\A+\B}{\sqrt{\det(\A+\B)}}\,,
\end{equation*}
where 
\begin{equation*}
\A \sharp \B:= \A^{\frac{1}{2}}\left(\A^{-\frac{1}{2}}\B\A^{-\frac{1}{2}}\right)^{\frac{1}{2}}\A^{\frac{1}{2}}\,.
\end{equation*}
is the geometric mean of positive definite matrices in the Kubo-Ando sense. See \cite{CHOIKIMLIM, GanKimMer2026} and the references therein\,.

\noindent The main purpose of this paper is to obtain an explicit decomposition of Kubo-Ando means on the cones of positive definite real symmetric, complex Hermitian, and quaternionic Hermitian matrices. We first prove that if $\Spec(\A^{-1}\B) = \left\{\lambda_{1}\,, \lambda_{2}\right\}$, then for any Kubo-Ando mean $\sigma$ with representing function $f$, we get
\begin{equation*}
\A \sigma \B = \frac{\lambda_{1}f(\lambda_{2}) - \lambda_{2}f(\lambda_{1})}{\lambda_{1} - \lambda_{2}}\A + \frac{f(\lambda_{1}) - f(\lambda_{2})}{\lambda_{1} - \lambda_{2}}\B\,.
\end{equation*}
However, the previous condition on the spectrum is not necessary for a fixed Kubo-Ando mean. Indeed, for all $\A,\B \in \mathscr{P}_{n}(\mathbb{D})$ (with $\mathbb{D} \in \left\{\mathbb{R}\,, \mathbb{C}\,, \mathbb{H}\right\}$), the arithmetic mean $\A \nabla \B$ is always a linear combination of $\A$ and $\B$. Nevertheless, we prove the following characterization\,.

\begin{theo*}

Let $\left(\A\,, \B\right) \in \mathscr{P}_{n}(\mathbb{D})^{2}$. Let $\sigma$ be a Kubo-Ando mean whose representing function $f$ is not affine. Then $\A \sigma \B$ is linearizable if and only if $\left|\Spec(\A^{-1}\B)\right| \leq 2$\,.

\end{theo*}

\noindent We then obtain a more general decomposition of $\A \sigma \B$, involving more than just $\A$ and $\B$. Let $\A,\B \in \mathscr{P}_{n}(\mathbb{D})$ be such that
\begin{equation*}
\Spec(\A^{-1}\B) = \left\{\lambda_{1}\,, \lambda_{2}\,, \ldots\,, \lambda_{r}\right\}\,.
\end{equation*}
Then
\begin{equation*}
\A \sigma \B=\A^{\frac{1}{2}}f(\X)\A^{\frac{1}{2}}\,, \qquad \left(\X = \A^{-\frac{1}{2}}\B\A^{-\frac{1}{2}}\right)\,,
\end{equation*}
and there exists a unique polynomial $\P$ of degree $r-1$ such that
\begin{equation*}
f(\X)=\P(\X)\,.
\end{equation*}
Using Lagrange interpolation, we obtain an explicit description of the coefficients of $\P$, which yields the following finite decomposition of $\A \sigma \B$\,.

\begin{theo*}

Let $\A\,, \B \in \mathscr{P}_{n}(\mathbb{D})$ such that $\left|\Spec(\A^{-1}\B)\right| = r$. Then 
\begin{equation*}
\A \sigma \B = \sum\limits_{k = 0}^{r-1} c_{k}\A\left(\A^{-1}\B\right)^{k}\,,
\end{equation*}
where the scalars $c_{k}$ are given by
\begin{equation*}
c_{k} = (-1)^{r-1-k}\sum\limits_{i = 1}^{r} f(\lambda_{i})\frac{e_{r-1-k}(\widehat{\lambda_{i}})}{\prod\limits_{j \neq i}(\lambda_{i} - \lambda_{j})}\,, \qquad \left(0 \leq k \leq r-1\right)\,,
\end{equation*}
with $\widehat{\lambda_{i}} = \left(\lambda_{1}\,, \ldots\,,\lambda_{i-1}\,, \lambda_{i+1}\,, \ldots\,, \lambda_{r}\right)$, and where $e_{i}$ is the elementary symmetric polynomial of degree $i$\,.

\end{theo*}

\noindent We also perform explicit computations in the case $n=3$. Using Cardano's formula together with the Cayley-Hamilton theorem, we obtain a description of $c_{0}\,, c_{1}\,, c_{2}$ using only $\det(\X)$, $\tr(\X)$, and $\tr(\X^{2})$\,.

\noindent We also prove an analogue of the previous decomposition theorem for alternative means under the additional assumption that $\A$ and $\B$ commute. This yields explicit decompositions without imposing any restriction on the cardinality of $\Spec(\X)$, and recovers and extends several results of \cite{CHOIKIMLIM}. More precisely, we establish the following theorem\,.

\begin{theo*}

Let $\hat{\sigma}_{f}$ be the alternative mean associated with a function $f$, i.e.
\begin{equation*}
\A \hat{\sigma}_{f} \B = f(\A^{-1} \sharp \B) \A f(\A^{-1} \sharp \B)\,, \qquad \left(\A\,, \B \in \mathscr{P}_{n}(\mathbb{D})\right)\,.
\end{equation*}
Suppose that $\A\,, \B \in \mathscr{P}_{n}(\mathbb{D})$ are such that $\A\B = \B\A$ and $r = \left|\Spec(\A^{-1}\B)\right|$. Then there exists coefficients $d_{0}\,, d_{1}\,, \ldots\,, d_{r-1}$, depending on $f$ and $\Spec(\A^{-1}\B)$ only, such that

\begin{equation*}
\A \hat{\sigma}_{f} \B = \sum\limits_{i=0}^{r-1} \sum\limits_{j=0}^{r-1} d_{i}d_{j}\A(\A^{-1}\B)^{\frac{i+j}{2}}\,.
\end{equation*}

\end{theo*}

\noindent Finally, we establish several compatibility results for the polynomial decomposition under the embedding
\begin{equation*}
\mathscr{P}_{n}(\mathbb{C}) \hookrightarrow \mathscr{P}_{2n}(\mathbb{R})\,,
\end{equation*}
introduced in \cite{FRANCOMERINO2}\,.

\section{Preliminaries}

Let $\mathbb{D} \in \left\{\mathbb{R}\,, \mathbb{C}\,, \mathbb{H}\right\}$. For all $x \in \mathbb{D}$, we denote by $\overline{x}$ the conjugate of $x$, i.e. for all $a\,, b\,, c\,, d \in \mathbb{R}$, we have
\begin{equation*}
\overline{a} = a\,, \qquad \overline{a+ib} = a-ib\,, \qquad \overline{a+ib+jc+kd} = a-ib-jc-kd\,.
\end{equation*} 
Let $\B: \mathbb{D}^{n} \times \mathbb{D}^{n} \to \mathbb{D}$ be the form given by
\begin{equation*}
\B(x\,, y) = \sum\limits_{i = 1}^{n}x_{i}\overline{y_{i}}\,.
\end{equation*}
The form $\B$ is Hermitian and positive-definite, i.e. for all $x\,, y \in \mathbb{D}^{n}$ and $z \in \mathbb{D}^{n}$ non-zero, we have
\begin{equation*}
\B(y\,, x) = \overline{\B(x\,, y)}\,, \qquad \B(z\,, z) > 0\,.
\end{equation*}

\begin{notation}

\noindent We denote by $\Mat_{n}(\mathbb{D})$ the set of $n$ by $n$ matrices with entries in $\mathbb{D}$, and by $\GL_{n}(\mathbb{D})$ the set of invertible matrices in $\Mat_{n}(\mathbb{D})$. Moreover, we denote by 
\begin{equation*}
\exp_{\mathbb{D}}: \Mat_{n}(\mathbb{D}) \ni \X \to \sum\limits_{k = 0}^{\infty} \frac{\X^{k}}{k!} \in \GL_{n}(\mathbb{D})
\end{equation*}
the corresponding exponential map\,.

\end{notation}

\noindent For all $\A \in \Mat_{n}(\mathbb{D})$, we denote by $\A^{*}$ the adjoint of $\A$ with respect to the form $\B$, i.e.
\begin{equation*}
\B(\A x\,, y) = \B(x\,, \A^{*}y)\,, \qquad \left(x\,, y \in \mathbb{D}^{n}\right)\,.
\end{equation*}
One can see easily see that $\A^{*} = \overline{\A}^{t}$. We denote by $\U_{\mathbb{D}}$ the subgroup of $\GL_{n}(\mathbb{D})$ given by
\begin{equation*}
\U_{\mathbb{D}} = \left\{g \in \GL_{n}(\mathbb{D})\,, gg^{*} = \Id_{n}\right\}\,.
\end{equation*}
and let $\mathfrak{u}_{\mathbb{D}}$ be the Lie algebra of $\U_{\mathbb{D}}$, i.e.
\begin{equation*}
\mathfrak{u}_{\mathbb{D}} = \left\{\X \in \Mat_{n}(\mathbb{D})\,, \X^{*} = -\X\right\}\,.
\end{equation*}
In particular, using the notations of \cite{KNAPP}, we have
\begin{equation*}
\U_{\mathbb{D}} = \begin{cases} \O(n) & \text{ if } \mathbb{D} = \mathbb{R} \\ \U(n) & \text{ if } \mathbb{D} = \mathbb{C} \\ \Sp(n) & \text{ if } \mathbb{D} = \mathbb{H} \end{cases}
\end{equation*}

\begin{definition}

Let $\A \in \Mat_{n}(\mathbb{D})$. We say that
\begin{itemize}
\item $\A$ is Hermitian if $\A = \A^{*}$\,,
\item $\A$ is positive (and write $\A > 0$) if $x^{*}\A x > 0$ for all non-zero $x \in \mathbb{D}^{n}$\,.
\item $\A$ is positive semidefinite (and write $\A \geq 0$) if $x^{*}\A x \geq 0$ for all non-zero $x \in \mathbb{D}^{n}$\,.
\end{itemize}

\end{definition}

\noindent We denote by $\mathfrak{p}_{n}(\mathbb{D})$ the set of hermitian matrices. In particular, we get
\begin{equation*}
\Mat_{n}(\mathbb{D}) = \mathfrak{u}_{\mathbb{D}} \oplus \mathfrak{p}_{n}(\mathbb{D})\,.
\end{equation*}
Let $\mathscr{P}^{0}_{n}(\mathbb{D})$ and $\mathscr{P}_{n}(\mathbb{D})$ the subsets of $\mathfrak{p}_{n}(\mathbb{D})$ given by
\begin{equation*}
\mathscr{P}^{0}_{n}(\mathbb{D}) = \left\{\X \in \mathfrak{p}_{n}(\mathbb{D})\,, \X \geq 0\right\}\,, \qquad \mathscr{P}_{n}(\mathbb{D}) = \left\{\X \in \mathfrak{p}_{n}(\mathbb{D})\,, \X > 0\right\}\,.
\end{equation*}
It is well-known (see \cite{BHATIA}) that $\exp_{\mathbb{D}}(\mathfrak{p}_{n}(\mathbb{D})) \subseteq \mathscr{P}_{n}(\mathbb{D})$ and that the corresponding map
\begin{equation*}
\exp_{\mathbb{D}}: \mathfrak{p}_{n}(\mathbb{D}) \to \mathscr{P}_{n}(\mathbb{D})
\end{equation*}
is bijective. We denote by $\log_{\mathbb{D}}: \mathscr{P}_{n}(\mathbb{D}) \to \mathfrak{p}_{n}(\mathbb{D})$ the inverse of $\exp_{\mathbb{D}}$\,.

\begin{notation}

\begin{itemize}
\item For all $s \in \mathbb{R}$ and $\X \in \mathscr{P}_{n}(\mathbb{D})$, we denote by $\X^{s}$ the element of $\mathscr{P}_{n}(\mathbb{D})$ given by
\begin{equation*}
\X^{s} := \exp_{\mathbb{D}}\left(s\log_{\mathbb{D}}(\X)\right)\,.
\end{equation*}
\item We denote by $\preceq$ the Loewner order on $\mathscr{P}_{n}(\mathbb{D})$. It is a partial order defined by
\begin{equation*}
\X \preceq \Y \quad \Leftrightarrow \quad \Y - \X \in \mathscr{P}^{0}_{n}(\mathbb{D})\,, \qquad \left(\X\,, \Y \in \mathscr{P}_{n}(\mathbb{D})\right)\,.
\end{equation*}
\end{itemize}

\end{notation}

\begin{remark}

\begin{enumerate}
\item For a matrix $\X \in \Mat_{n}(\mathbb{D})$, we denote by $\Spec(\X)$ the spectrum of $\X$, i.e. the set of eigenvalues. In the quaternionic case, we will work with right eigenvalues, i.e. $\lambda \in \mathbb{H}$ such that
\begin{equation*}
\X v = v\lambda
\end{equation*}
for non-zero vector $v \in \mathbb{H}^{n}$. As explained in \cite{ZHANG}, in the quaternionic setting, if $\lambda \in \Spec(\X)$, then $\gamma^{-1}\lambda\gamma \in \Spec(\X)$ for all $\gamma \in \mathbb{H}^{*}$. However, for hermitian matrices, we have $\Spec(\X) \subseteq \mathbb{R}$, and a matrix $\X \in \mathfrak{p}_{n}(\mathbb{D})$ has at most $n$ different eigenvalues\,.
\item For all $\X\,, \Y \in \mathscr{P}_{n}(\mathbb{D})$, we have $\X^{-\frac{1}{2}}\Y\X^{-\frac{1}{2}} \in \mathscr{P}_{n}(\mathbb{D})$. Even if $\X^{-1}\Y$ is not in $\mathscr{P}_{n}(\mathbb{D})$ in general, we have
\begin{equation*}
\Spec(\X^{-\frac{1}{2}}\Y\X^{-\frac{1}{2}}) = \Spec(\X^{-1}\Y)\,.
\end{equation*}
Indeed,
\begin{equation*}
\X^{-1}\Y = \X^{-\frac{1}{2}}\left(\X^{-\frac{1}{2}}\Y\X^{-\frac{1}{2}}\right)\X^{\frac{1}{2}}\,,
\end{equation*}
so $\X^{-1}\Y$ and $\X^{-\frac{1}{2}}\Y\X^{-\frac{1}{2}}$ are similar\,.
\item (Spectral decomposition) Let $\X \in \mathscr{P}_{n}(\mathbb{D})$. Then $\X$ can be written of the form
\begin{equation*}
\X = \U\Lambda\U^{*}\,, 
\end{equation*}
with $\U \in \U_{\mathbb{D}}$ and $\Lambda = \diag\left(\lambda_{1}\,, \ldots\,, \lambda_{n}\right)$, where $\lambda_{i} \in \left(0\,, \infty\right)$ are the eigenvalues of $\X$\,.
\item The spectral decomposition plays an important role in functional calculus. Indeed, if $f: \left(0\,, \infty\right) \to \mathbb{R}$ is a continuous function, then $f$ can be extended to $\mathscr{P}_{n}(\mathbb{D})$ by
\begin{equation*}
f(\X) = \U f(\Lambda)\U^{*}\,,
\end{equation*}
with $f(\Lambda) = \diag\left(f(\lambda_{1})\,, \ldots\,, f(\lambda_{n})\right)$\,.
\end{enumerate}

\label{RemarkPreliminaries}

\end{remark}

\noindent We can now define the notion of Kubo-Ando means on $\mathscr{P}_{n}(\mathbb{D})$ (see \cite{KUBOANDO})\,.

\begin{definition}

A Kubo-Ando mean on $\mathscr{P}_{n}(\mathbb{D})$ is a continuous map
\begin{equation*}
\sigma: \mathscr{P}_{n}(\mathbb{D}) \times \mathscr{P}_{n}(\mathbb{D}) \mapsto \mathscr{P}_{n}(\mathbb{D})
\end{equation*}
such that
\begin{enumerate}
\item If $\A \preceq \B$ and $\C \preceq \D$, then $\A \sigma \C \preceq \B \sigma \D$\,,
\item For all $g \in \GL_{n}(\mathbb{D})$ and $\A\,, \B \in \mathscr{P}_{n}(\mathbb{D})$, we have
\begin{equation*}
g\left(\A \sigma \B\right)g^{*} = \left(g \A g^{*}\right) \sigma \left(g \B g^{*}\right)\,.
\end{equation*}
\item $\Id_{n} \sigma \Id_{n} = \Id_{n}$\,.
\end{enumerate}

\end{definition}

\noindent One central result of \cite{KUBOANDO} is summarized in the next theorem\,.

\begin{theo}

There exists a bijection between Kubo-Ando $\sigma$ on $\mathscr{P}_{n}(\mathbb{D})$ and the set of operator monotone functions $f: \left(0\,, +\infty\right) \to \left(0\,, +\infty\right)$ such that $f(1) = 1$. More precisely, for all mean $\sigma$ on $\mathscr{P}_{n}(\mathbb{D})$, there exists a unique operator monotone function $f: \left(0\,, +\infty\right) \to \left(0\,, +\infty\right)$, with $f(1) = 1$, satisfying
\begin{equation}
\A \sigma \B = \A^{\frac{1}{2}} f\left(\A^{-\frac{1}{2}}\B\A^{-\frac{1}{2}}\right)\A^{\frac{1}{2}}\,, \qquad \left(\A\,, \B \in \mathscr{P}_{n}(\mathbb{D})\right)\,.
\label{EquationKuboAndoSigma}
\end{equation}

\label{TheoremKuboAndo}

\end{theo}

\noindent The main goal of this paper is to obtain an explicit decomposition of Kubo-Ando means in terms of the spectrum of $\A^{-1}\B$. More precisely, we show that $\A \sigma \B$ can be written as a finite linear combination of matrices of the form $\A\left(\A^{-1}\B\right)^k$, whose coefficients depend only on the representing function $f$ of $\sigma$ and the eigenvalues of $\A^{-1}\B$\,.

\section{On the linearization of Kubo-Ando means}

We start with the following proposition, extending Theorem 5.3 of \cite{FRANCOMERINO2}\,.

\begin{proposition}

Let $\sigma$ be a Kubo-Ando mean with representing function $f$, and let $\A\,, \B \in \mathscr{P}_{n}(\mathbb{D})$ such that $\left|\Spec(\A^{-1}\B)\right| = 2$. Then $\A \sigma \B$ is linearizable. More precisely, if $\Spec(\A^{-1}\B) = \left\{\lambda_{1}\,, \lambda_{2}\right\}$, we have
\begin{equation*}
\A \sigma \B = \frac{\lambda_{1}f(\lambda_{2}) - \lambda_{2}f(\lambda_{1})}{\lambda_{1} - \lambda_{2}}\A + \frac{f(\lambda_{1}) - f(\lambda_{2})}{\lambda_{1} - \lambda_{2}}\B\,.
\end{equation*}

\label{PropositionSpectrumTwo}

\end{proposition}

\begin{proof}

Let $\X = \A^{-\frac{1}{2}}\B\A^{-\frac{1}{2}}$. We have $\Spec(\A^{-1}\B) = \left\{\lambda_{1}\,, \lambda_{2}\right\}$. Therefore, there exists a unique degree one polynomial $r(t) = a + bt$ such that $r(\lambda_{i}) = f(\lambda_{i})$. More precisely, $a$ and $b$ are the unique solution of the system
\begin{equation*}
\begin{cases} a + b\lambda_{1} = f(\lambda_{1}) \\ a + b\lambda_{2} = f(\lambda_{2}) \end{cases}
\end{equation*}
i.e.
\begin{equation*}
a = \frac{\lambda_{1}f(\lambda_{2}) - \lambda_{2}f(\lambda_{1})}{\lambda_{1} - \lambda_{2}}\,, \qquad b = \frac{f(\lambda_{1}) - f(\lambda_{2})}{\lambda_{1} - \lambda_{2}}\,.
\end{equation*}
Since $\X \in \mathscr{P}_{n}(\mathbb{D})$ and $\Spec(\X) = \left\{\lambda_{1}\,, \lambda_{2}\right\}$, there exists $1 \leq m < n$ and $\U \in \U_{\mathbb{D}}$ such that $\X = \U\Lambda\U^{*}$, with $\Lambda = \diag\left(\lambda_{1}\Id_{m}\,, \lambda_{2}\Id_{n-m}\right)$. In particular, by functional calculus, we get
\begin{eqnarray*}
r(\X) & = & \U r(\Lambda)\U^{*} = \U\diag\left(r(\lambda_{1})\Id_{m}\,, r(\lambda_{2})\Id_{n-m}\right)\U^{*} \\
& = & \U\diag\left(f(\lambda_{1})\Id_{m}\,, f(\lambda_{2})\Id_{n-m}\right)\U^{*} =  \U f(\Lambda)\U^{*} = f(\X)\,.
\end{eqnarray*}
Finally, using that $\A^{\frac{1}{2}}\X\A^{\frac{1}{2}} = \B$, we get
\begin{equation*}
\A \sigma \B = \A^{\frac{1}{2}}f(\X)\A^{\frac{1}{2}} = \A^{\frac{1}{2}}\left(a\Id_{n} + b\X\right)\A^{\frac{1}{2}} = a\A + b\B\,,
\end{equation*}
and the proposition follows\,.

\end{proof}

\begin{remark}

\begin{enumerate}
\item In Proposition \ref{PropositionSpectrumTwo}, the condition $\left|\Spec(\A^{-1}\B)\right| = 2$ could be replaced by $\left|\Spec(\A^{-1}\B)\right| \leq  2$. Indeed, if $\left|\Spec(\A^{-1}\B)\right| = 1$, then $\B = \lambda\A$, and therefore
\begin{equation*}
A \sigma \B = f(\lambda)\A\,.
\end{equation*}
\item It is easy to see that the condition $\left|\Spec(\A^{-1}\B)\right| \leq 2$ is sufficient but not necessary. Indeed, let $\nabla$ be the arithmetic mean on $\mathscr{P}_{n}(\mathbb{D})$. Then for all pairs $\left(\A\,, \B\right)$, $\A \nabla \B$ is linearizable. However, we get the following theorem\,.
\end{enumerate}

\end{remark}

\begin{theo}

Let $\left(\A\,, \B\right) \in \mathscr{P}_{n}(\mathbb{D})^{2}$. Let $\sigma$ be a Kubo-Ando mean whose representing function $f$ is not affine. Then $\A \sigma \B$ is linearizable if and only if $\left|\Spec(\A^{-1}\B)\right| \leq 2$\,.

\label{TheoremSectionTwo}

\end{theo}

\begin{proof}

Assume that $\A \sigma \B$ is linearizable. In particular, it follows from Equation \ref{EquationKuboAndoSigma} that there exists $a\,, b \in \mathbb{R}$ such that
\begin{equation*}
\A^{\frac{1}{2}}f(\X)\A^{\frac{1}{2}} = a\A + b\B\,,
\end{equation*}
with $\X = \A^{-\frac{1}{2}}\B\A^{-\frac{1}{2}}$. By multiplying the previous equation by $\A^{-\frac{1}{2}}$ on both sides, we get
\begin{equation}
f(\X) = a\Id_{n} + b\X\,.
\label{SquareRootX}
\end{equation}
Let $\lambda \in \Spec(\X)$. Then it follows from Equation \eqref{SquareRootX} that
\begin{equation*}
f(\lambda) = a + b\lambda\,.
\end{equation*}
Using that $f$ is not affine, it follows that $f$ is strictly concave. Therefore, the equation $f(\lambda) = a + b\lambda$ has at most two solutions, i.e. $\left|\Spec(\X)\right| = \left|\Spec(\A^{-1}\B)\right| \leq 2$\,.

\noindent The converse follows from Proposition \ref{PropositionSpectrumTwo}\,.
\end{proof}

\section{Decomposition of Kubo-Ando means}

\label{SectionDecompositionKAMeans}

We first give an analogue of Proposition \ref{PropositionSpectrumTwo}. Let $\sigma$ be a Kubo-Ando mean with representing function $f$. 

\begin{proposition}

Let $\X \in \mathscr{P}_{n}(\mathbb{D})$ such that $\left|\Spec(\X)\right| = r$. There exists a unique polynomial $\P$ of degree $r-1$ such that $f(\X) = \P(\X)$\,.

\label{PropositionExsitenceP}

\end{proposition}

\begin{proof}

Let $\Spec(\X) = \left\{\lambda_{1}\,, \ldots\,, \lambda_{r}\right\}$. The matrix $\X$ can be written as $\X = \U\Lambda\U^{*}$, where $\U \in \U_{\mathbb{D}}$ and $\Lambda = \diag\left(\lambda_{i}\Id_{n_{i}}\right)$, where $n_{i} \in \mathbb{N}^{*}$ is the geometric multiplicity of $\lambda_{i}$. Let $\V(\lambda) := \V(\lambda_{1}\,, \ldots\,, \lambda_{r}) \in \Mat_{r}(\mathbb{R})$ be the Vandermonde matrix, i.e.
\begin{equation*}
\V(\lambda)_{i,j} = \lambda^{j-1}_{i}\,, \qquad \left(1 \leq i\,, j \leq r\right)\,.
\end{equation*}
It is well-known that $\det(\V_{\lambda}) = \prod\limits_{1 \leq i < j \leq r} \left(\lambda_{i} - \lambda_{j}\right)$, i.e. $\det(\V(\lambda)) \neq 0$. Let $\F(\lambda)$ be the vector of $\mathbb{R}^{r}$ given by $\F(\lambda) = \left(f(\lambda_{1})\,, \ldots\,, f(\lambda_{r})\right)^{t}$. We denote by $\C = \left(c_{0}\,, c_{1}\,, \ldots\,, c_{r-1}\right)$ the unique solution of $\V(\lambda)\C = \F(\lambda)$. In particular, we have
\begin{equation*}
\sum\limits_{k = 0}^{r-1}  c_{k}\lambda^{k}_{i} = f(\lambda_{i})\,, \qquad \left(1 \leq i \leq r\right)\,,
\end{equation*}
and let $\P$ be the polynomial given by $\P(t) = \sum\limits_{k = 0}^{r-1} c_{k}t^{k}$. Using that $\P(\lambda_{i}) = f(\lambda_{i})$ for all $i$, it follows from functional calculus that
\begin{eqnarray*}
\P(\X) & = & \U\P(\Lambda)\U^{*} = \U\diag\left(\P(\lambda_{1})\Id_{n_{1}}\,, \ldots\,, \P(\lambda_{r})\Id_{n_{r}}\right)\U^{*} \\
& = & \U\diag\left(f(\lambda_{1})\Id_{n_{1}}\,, \ldots\,, f(\lambda_{r})\Id_{n_{r}}\right)\U^{*} = f(\X)\,.
\end{eqnarray*}

\end{proof}

\begin{remark}

The coefficients $c_{0}\,, c_{1}\,, \ldots\,, c_{r-1}$ obtained in the previous proof can be determined explicitly. For all $1 \leq i \leq r$, we denote by $\V_{i}(\F(\lambda))$ the matrix obtained by replacing the $i$-th column of $\V(\lambda)$ by $\F(\lambda)$. Then for all $0 \leq i \leq r-1$, we have
\begin{equation*}
c_{i} = \frac{\det(\V_{i+1}(\F(\lambda)))}{\det(\V(\lambda))}\,.
\end{equation*}

\end{remark}

\noindent We can even use Lagrange interpolation. Indeed the unique polynomial $\P$ of degree $r-1$ satisfying $\P(\lambda_{i}) = f(\lambda_{i})\,, 1 \leq i \leq r$ is given by
\begin{equation}
\P(t) = \sum\limits_{i = 1}^{r} f(\lambda_{i}) \prod\limits_{j \neq i} \frac{t - \lambda_{j}}{\lambda_{i} - \lambda_{j}}\,.
\label{PolynomialPT}
\end{equation}

\begin{notation}

\begin{itemize}
\item Let $\lambda = \left(\lambda_{1}\,, \ldots\,, \lambda_{r}\right)$. For all $1 \leq i \leq r$, we denote by $\widehat{\lambda_{i}}$ the vector of $\mathbb{R}^{r-1}$ obtained by removing $\lambda_{i}$ from $\lambda$\,.
\item For all $0 \leq m \leq r-1$, we denote by $e_{m}$ the elementary symmetric polynomial, i.e.
\begin{equation*}
e_{m}\left(\gamma_{1}\,, \ldots\,, \gamma_{r-1}\right) = \sum\limits_{1 \leq i_{1} < i_{2} < \ldots < i_{m} \leq r-1} \gamma_{i_{1}}\gamma_{i_{2}} \ldots \gamma_{i_{m}}\,,
\end{equation*}
with $e_{0} = 1$\,.
\end{itemize}

\end{notation}

\noindent We now simplify Equation \eqref{PolynomialPT}. We have
\begin{equation*}
\prod\limits_{j \neq i}(t - \lambda_{j}) = \sum\limits_{k = 0}^{r-1} (-1)^{r-1-k}e_{r-1-k}(\widehat{\lambda_{i}})t^{k}\,,
\end{equation*}
i.e.
\begin{eqnarray*}
\P(t) & = & \sum\limits_{i = 1}^{r} f(\lambda_{i}) \prod\limits_{j \neq i} \frac{t - \lambda_{j}}{\lambda_{i} - \lambda_{j}} = \sum\limits_{i = 1}^{r} f(\lambda_{i})\left(\sum\limits_{k = 0}^{r-1} (-1)^{r-1-k}\frac{e_{r-1-k}(\widehat{\lambda_{i}})}{\prod\limits_{j \neq i}(\lambda_{i} - \lambda_{j})}t^{k}\right) \\ 
& = & \sum\limits_{k = 0}^{r-1} (-1)^{r-1-k}\left(\sum\limits_{i = 1}^{r} f(\lambda_{i})\frac{e_{r-1-k}(\widehat{\lambda_{i}})}{\prod\limits_{j \neq i}(\lambda_{i} - \lambda_{j})}\right)t^{k}
\end{eqnarray*}
Finally, we get 
\begin{equation}
c_{k} = (-1)^{r-1-k}\sum\limits_{i = 1}^{r} f(\lambda_{i})\frac{e_{r-1-k}(\widehat{\lambda_{i}})}{\prod\limits_{j \neq i}(\lambda_{i} - \lambda_{j})}\,, \qquad \left(0 \leq k \leq r-1\right)\,.
\label{CoefficientsCK}
\end{equation}

\begin{theo}

Let $\A\,, \B \in \mathscr{P}_{n}(\mathbb{D})$ such that $\left|\Spec(\A^{-1}\B)\right| = r$. Then
\begin{equation}
\A \sigma \B = c_{0}\A + \sum\limits_{k = 1}^{r-1} c_{k}\left(\B\A^{-1}\right)^{k-1}\B\,.
\label{DecompositionASigmaB}
\end{equation}
where the coefficients $c_{k}$ depend only on $f$ and the eigenvalues of $\A^{-1}\B$, and are given in Equation \eqref{CoefficientsCK}\,.

\label{TheoremGeneralDecomposition}

\end{theo}

\begin{proof}

Let $\X = \A^{-\frac{1}{2}}\B\A^{-\frac{1}{2}}$. Using Remark \ref{RemarkPreliminaries}, we have $\Spec(\A^{-1}\B) = \Spec(\X)$. As explained in Proposition \ref{PropositionExsitenceP}, we have $f(\X) = \P(\X)$, with $\P(\X) = \sum\limits_{k = 0}^{r-1}c_{k}\X^{k}$. Then it follows from Proposition \ref{PropositionExsitenceP} that
\begin{equation*}
\A \sigma \B = \A^{\frac{1}{2}}f(\X)\A^{\frac{1}{2}} = \sum\limits_{k = 0}^{r-1}c_{k}\A^{\frac{1}{2}}\X^{k}\A^{\frac{1}{2}}\,.
\end{equation*}
Using that for all $1 \leq k \leq r-1$, we get
\begin{equation*}
\A^{\frac{1}{2}}\X^{k}\A^{\frac{1}{2}} = \A^{\frac{1}{2}}\left(\A^{-\frac{1}{2}}\B\A^{-\frac{1}{2}} \cdot \A^{-\frac{1}{2}}\B\A^{-\frac{1}{2}} \cdot \ldots \cdot \A^{-\frac{1}{2}}\B\A^{-\frac{1}{2}}\right)\A^{\frac{1}{2}} = (\B\A^{-1})^{k-1}\B\,,
\end{equation*}
therefore
\begin{equation*}
\A \sigma \B = \sum\limits_{k = 0}^{r-1}c_{k}\A^{\frac{1}{2}}\X^{k}\A^{\frac{1}{2}} = c_{0}\A + \sum\limits_{k = 1}^{r-1}c_{k}(\B\A^{-1})^{k-1}\B\,.
\end{equation*}

\end{proof}

\begin{remark}

\begin{enumerate}
\item The coefficients $c_{k}$ in the decomposition of $\A \sigma \B$ given in Equation \eqref{DecompositionASigmaB} only depends on $f$ and $\Spec(\A^{-1}\B)$. In the next section, we will show that in the case $n = 3$, the roots $\lambda_{1}\,, \lambda_{2}\,, \lambda_{3}$ only depends on $\det(\X)\,, \tr(\X)$, and $\tr(\X^{2})$\,.
\item The equality obtained in Equation \eqref{DecompositionASigmaB} can be rewritten as
\begin{equation}
\A \sigma \B = \sum\limits_{i = 0}^{r-1} c_{i}\A\left(\A^{-1}\B\right)^{i}\,.
\label{EquationASigmaBSecondVersion}
\end{equation}
\item Assume that $\left|\Spec(\X)\right| = r$. Then it follows that the set of matrices  
\begin{equation*}
\left\{\Id_{n}\,, \X\,, \ldots\,, \X^{r-1}\right\}
\end{equation*}
is linearly independent. Therefore, Theorem \ref{TheoremSectionTwo} can be obtained directly using Equation \eqref{EquationASigmaBSecondVersion}\,.
\end{enumerate}

\end{remark}

\section{Explicit computations in $\mathscr{P}_{3}(\mathbb{D})$}

Let $\sigma$ be a Kubo-Ando mean on $\mathscr{P}_{3}(\mathbb{D})$. Let $\A\,, \B$ be two matrices in $\mathscr{P}_{3}(\mathbb{D})$, and let $\X = \A^{-\frac{1}{2}}\B\A^{-\frac{1}{2}}$. In this section, we will give a description of the coefficients appearing in the decomposition of $\A \sigma \B$ only using $\det(\X)\,, \tr(\X)\,,$ and $\tr(\X^{2})$. We will distinguish three different cases: $\left|\Spec(\X)\right| = 1\,, 2\,,$ or $3$\,.

\noindent We start with the easiest case\,.

\begin{theo}

Suppose that $\left|\Spec(\X)\right| = 1$. Then there exists $a \in \left(0\,, \infty\right)$ such that $\B = a\A$, and we have
\begin{equation*}
\A \sigma \B = f(a)\A\,.
\end{equation*}

\end{theo}

\begin{proof}

Suppose that $\left|\Spec(\X)\right| = 1$. Therefore, there exists $a \in \mathbb{D}$ such that $\X = a\Id_{3}$. Moreover, using that $\X \in \mathscr{P}_{n}(\mathbb{D})$, it follows that $a > 0$. From the equation 
\begin{equation*}
\A^{-\frac{1}{2}}\B\A^{-\frac{1}{2}} = a\Id_{3}\,,
\end{equation*}
we get $\B = a\A$, so $\X = \A^{-\frac{1}{2}}\left(a\A\right)\A^{-\frac{1}{2}} = a\Id_{3}$. Finally, using that $f(a\Id_{3}) = f(a)\Id_{3}$, we get
\begin{equation*}
\A \sigma \B = \A^{\frac{1}{2}}f(\X)\A^{\frac{1}{2}} = \A^{\frac{1}{2}}f(a\Id_{3})\A^{\frac{1}{2}} = \A^{\frac{1}{2}}\left(f(a)\Id_{3}\right)\A^{\frac{1}{2}} = f(a)\A\,.
\end{equation*}

\end{proof}

\begin{theo}

Suppose that $\left|\Spec(\X)\right| = 2$. Suppose that $\lambda_{1}$ has geometric multiplicity $2$ and $\lambda_{2}$ has geometric multiplicity $1$. Then
\begin{equation*}
\A \sigma \B = \frac{\lambda_{1}f(\lambda_{2}) - \lambda_{2}f(\lambda_{1})}{\lambda_{1} - \lambda_{2}}\A + \frac{f(\lambda_{1}) - f(\lambda_{2})}{\lambda_{1} - \lambda_{2}}\B\,,
\end{equation*}
with
\begin{equation}
\lambda_{1} = \frac{2\tr(\X) + a\sqrt{6\tr(\X^{2}) - 2\tr(\X)^{2}}}{6}\,, \qquad \lambda_{2} = \frac{\tr(\X) - 2a\sqrt{6\tr(\X^{2}) - 2\tr(\X)^{2}}}{3}\,,
\label{LambdaOneLambdaTwo}
\end{equation}
and where $a \in \left\{-1\,, 1\right\}$ is chosen so that $\det(\X) = \lambda^{2}_{1}\lambda_{2}$\,.

\end{theo}

\begin{proof}

It follows from Proposition \ref{PropositionSpectrumTwo} that $\A \sigma \B$ is of the form
\begin{equation*}
\A \sigma \B = \frac{\lambda_{1}f(\lambda_{2}) - \lambda_{2}f(\lambda_{1})}{\lambda_{1} - \lambda_{2}}\A + \frac{f(\lambda_{1}) - f(\lambda_{2})}{\lambda_{1} - \lambda_{2}}\B\,,
\end{equation*}
with $\Spec(\A^{-\frac{1}{2}}\B\A^{-\frac{1}{2}}) = \left\{\lambda_{1}\,, \lambda_{2}\right\}$. We need to prove that $\lambda_{1}$ and $\lambda_{2}$ are the ones given in Equation \eqref{LambdaOneLambdaTwo}. To simplify, we will use the following notations:
\begin{equation*}
\T_{1} = \tr(\X)\,, \qquad \T_{2} = \tr(\X^{2})\,.
\end{equation*}
Then
\begin{equation*}
\begin{cases} \T_{1} & = 2\lambda_{1} + \lambda_{2} \\ \T_{2} & = 2\lambda^{2}_{1} + \lambda^{2}_{2} \end{cases}
\end{equation*}
and $\det(\X) = \lambda^{2}_{1}\lambda_{2}$ From the first equation, we get $\lambda_{2} = \T_{1} - 2\lambda_{1}$. Substituting $\lambda_{2}$ in the second equation, we get
\begin{equation*}
\T_{2} = 2\lambda^{2}_{1} + \left(\T_{1} - 2\lambda_{1}\right)^{2} = 6\lambda^{2}_{1} - 4\T_{1}\lambda_{1} + \T^{2}_{1}\,.
\end{equation*}
Thus
\begin{equation*}
6\lambda^{2}_{1} - 4\T_{1}\lambda_{1} + \left(\T^{2}_{1} - \T_{2}\right) = 0\,,
\end{equation*}
i.e.
\begin{equation*}
\lambda_{1} = \frac{2\T_{1} + a\sqrt{6\T_{2} - 2\T^{2}_{1}}}{6}\,, \qquad a \in \left\{-1\,, 1\right\}\,.
\end{equation*}
Therefore
\begin{equation*}
\lambda_{2} = \T_{1} - 2\lambda_{1} = \frac{\T_{1} - 2a\sqrt{6\T_{2} - 2\T^{2}_{1}}}{3}\,.
\end{equation*}
and the sign $a$ is chosen so that $\det(\X) = \lambda^{2}_{1}\lambda_{2}$.

\end{proof}

\medskip

\noindent Now assume that $\X \in \mathscr{P}_{3}(\mathbb{D})$ is such that $\left|\Spec(\X)\right|=3$. By the Cayley-Hamilton theorem, we have
\begin{equation*}
\X^{3} - \tr(\X)\X^{2} + \frac{\tr(\X)^{2}-\tr(\X^{2})}{2}\X - \det(\X)\Id_{3} = 0\,,
\end{equation*}
In other words, the eigenvalues of $\X$ are the roots of
\begin{equation*}
x^{3} - \tr(\X)x^{2} + \frac{\tr(\X)^{2}-\tr(\X^{2})}{2}x - \det(\X) = 0\,.
\end{equation*}
We will use the following notation:
\begin{equation*}
\T = \tr(\X)\,, \qquad \Q = \tr(\X^{2})\,, \qquad \D=\det(\X)\,, \qquad \S=\frac{\T^{2} - \Q}{2}\,.
\end{equation*}
Then the characteristic polynomial is
\begin{equation*}
x^{3} -\T x^{2} + \S x - \D\,.
\end{equation*}
Using the change of variable $x = y+\frac{\T}{3}$, we get the equation
\begin{equation*}
y^{3} + py + q = 0\,,
\end{equation*}
where
\begin{equation*}
\p = \S - \frac{\T^{2}}{3}\,, \qquad \q = -\frac{2\T^{3}}{27} + \frac{\T\S}{3}-\D\,.
\end{equation*}
Let
\begin{equation*}
\Delta = \left(\frac{\q}{2}\right)^{2} + \left(\frac{\p}{3}\right)^{3}\,, \qquad  \omega=\frac{-1+i\sqrt{3}}{2}\,.
\end{equation*}
By Cardano's formula, the three roots are
\begin{equation}
\lambda_{k} = \frac{\T}{3} + \omega^{k}\sqrt[3]{-\frac{\q}{2}+\sqrt{\Delta}} + \omega^{2k}\sqrt[3]{-\frac{\q}{2}-\sqrt{\Delta}}\,, \qquad \left(k = 0\,, 1\,, 2\right)\,.
\label{EquationRootsLambdak}
\end{equation}
Therefore, when $\left|\Spec(\X)\right| = 3$, the three distinct spectral values of $\X$ are explicitly determined by $\tr(\X)$, $\tr(\X^{2})$, and $\det(\X)$\,.

\begin{remark}

Since $\X$ is Hermitian, we have $\Spec(\X) \subseteq \mathbb{R}$. In this case, it follows that the discriminant
\begin{equation*}
\Delta = \left(\frac{\q}{2}\right)^{2}+\left(\frac{\p}{3}\right)^{3}
\end{equation*}
is negative. In particular,
\begin{equation*}
-\frac{\q}{2}+\sqrt{\Delta} \qquad \text{and} \qquad -\frac{\q}{2} - \sqrt{\Delta}
\end{equation*}
are complex conjugates. Writing
\begin{equation*}
-\frac{\q}{2}+\sqrt{\Delta} = \rho e^{i\theta}\,,
\end{equation*}
with
\begin{equation*}
\rho = \left|-\frac{\q}{2} + \sqrt{\Delta}\right| = \left|-\frac{\q}{2} + i\sqrt{-\Delta}\right| = \left(-\frac{\p}{3}\right)^{3/2}\,.
\end{equation*}
it follows that the roots given in Equation \eqref{EquationRootsLambdak} can be rewritten as
\begin{equation*}
\lambda_{k} = \frac{\T}{3} + 2\sqrt{-\frac{\p}{3}}\cos\left(\frac{\theta+2k\pi}{3}\right)\,, \qquad \left(k=0\,, 1\,, 2\right)\,,
\end{equation*}
with
\begin{equation*}
\theta = \arccos\left(-\frac{\q}{2}\left(-\frac{3}{\p}\right)^{3/2}\right)\,.
\end{equation*}
It is now easier to see why the roots $\lambda_{k}$ are real\,.

\label{RemarkCardano}

\end{remark}

\begin{theo}

Let $\A\,, \B \in \mathscr{P}_{3}(\mathbb{D})$ such that $\left|\Spec(\A^{-\frac{1}{2}}\B\A^{-\frac{1}{2}})\right| = 3$. Then
\begin{equation*}
\A \sigma \B = c_{0}\A + c_{1}\B + c_{2}\B\A^{-1}\B\,,
\end{equation*}
where
\begin{align*}
c_{0} & = \frac{f(\lambda_{1})\lambda_{2}\lambda_{3}}{(\lambda_{1}-\lambda_{2})(\lambda_{1}-\lambda_{3})} - \frac{f(\lambda_{2})\lambda_{1}\lambda_{3}}{(\lambda_{1}-\lambda_{2})(\lambda_{2}-\lambda_{3})} + \frac{f(\lambda_{3})\lambda_{1}\lambda_{2}}{(\lambda_{1}-\lambda_{3})(\lambda_{2}-\lambda_{3})}\,, \\ 
c_{1} & = -\frac{f(\lambda_{1})(\lambda_{2}+\lambda_{3})}{(\lambda_{1}-\lambda_{2})(\lambda_{1}-\lambda_{3})} + \frac{f(\lambda_{2})(\lambda_{1}+\lambda_{3})}{(\lambda_{1}-\lambda_{2})(\lambda_{2}-\lambda_{3})} - \frac{f(\lambda_{3})(\lambda_{1}+\lambda_{2})}{(\lambda_{1}-\lambda_{3})(\lambda_{2}-\lambda_{3})}\,, \\ 
c_{2} & =\frac{f(\lambda_{1})}{(\lambda_{1}-\lambda_{2})(\lambda_{1}-\lambda_{3})} - \frac{f(\lambda_{2})}{(\lambda_{1}-\lambda_{2})(\lambda_{2}-\lambda_{3})} + \frac{f(\lambda_{3})}{(\lambda_{1}-\lambda_{3})(\lambda_{2}-\lambda_{3})}\,.
\end{align*}
and where $\lambda_{1}\,, \lambda_{2}\,, \lambda_{3}$ are given in Remark \ref{RemarkCardano} and depend only on $\det(\X)\,, \tr(\X)\,,$ and $\tr(\X^{2})$\,.

\end{theo}

\begin{proof}

Follows from Theorem \ref{TheoremGeneralDecomposition} and Remark \ref{RemarkCardano}\,.

\end{proof}

\begin{remark}

\begin{enumerate}

\item Although the coefficients $c_{0}\,, c_{1}\,, c_{2}$ are expressed in terms of the eigenvalues of $\X$, they can be computed easily. Indeed, the eigenvalues of a $3$ by $3$ matrix can be obtained explicitly from the coefficients of its characteristic polynomial using Cardano's formula. Consequently, the previous decomposition yields an effective procedure for computing any Kubo-Ando mean, and can be implemented directly in symbolic computation software such as Matlab or Python\,.

\noindent This is useful to reduce computation complexity for means of matrices. This is especially important for means that involve logarithms, which are computationally intensive.
\item Similar formulas can be obtained for $4$ by $4$ matrices. Indeed, if $\X \in \mathscr{P}_{4}(\mathbb{D})$, then the Cayley-Hamilton theorem gives \small{
\begin{equation*}
\X^{4} - \tr(\X)\X^{3} + \frac{\tr(\X)^{2}-\tr(\X^{2})}{2}\X^{2} -\frac{\tr(\X)^{3}-3\tr(\X)\tr(\X^{2})+2\tr(\X^{3})}{6}\X + \det(\X)\Id_{4} = 0\,.
\end{equation*}}
Therefore, the eigenvalues of $\X$ are roots of a quartic polynomial whose coefficients depend only on
\begin{equation*}
\tr(\X)\,, \qquad \tr(\X^{2})\,, \qquad \tr(\X^{3})\,, \qquad \det(\X)\,.
\end{equation*}
By Ferrari's formula, these roots can be expressed explicitly in terms of these quantities\,.
\end{enumerate}

\end{remark}

\section{Compatibility with the embedding $\mathscr{P}_{n}(\mathbb{C}) \hookrightarrow \mathscr{P}_{2n}(\mathbb{R})$}

We now recall some results of \cite{FRANCOMERINO2}. Any matrix $\X \in \Mat_{n}(\mathbb{C})$ can be written as $\X = \A + i\B$, with $\A\,, \B \in \Mat_{n}(\mathbb{R})$. Let $\Psi$ be the map given by
\begin{equation*}
\Psi: \Mat_{n}(\mathbb{C}) \ni \X = \A + i\B \to \begin{pmatrix} \A & \B \\ -\B & \A \end{pmatrix} \in \Mat_{2n}(\mathbb{R})\,.
\end{equation*}
The map $\Psi$ is a monomorphism of algebras and such that
\begin{equation*}
\Psi(\X^{*}) = \Psi(\X)^{t}\,, \qquad \left(\X \in \Mat_{n}(\mathbb{C})\right)\,.
\end{equation*}
As explained in \cite[Proposition~2.5]{FRANCOMERINO2}, we have $\Psi(\mathscr{P}_{n}(\mathbb{C})) \subseteq \mathscr{P}_{2n}(\mathbb{R})$. More precisely, we have
\begin{equation*}
\Psi(\mathscr{P}_{n}(\mathbb{C})) = \left\{\X \in \mathscr{P}_{2n}(\mathbb{R})\,, \X\K_{n} = \K_{n}\X\right\}\,,
\end{equation*}
where $\K_{n} = \Psi(i\Id_{n})$. We denote by $\widetilde{\mathscr{P}}_{n}(\mathbb{C})$ the subset $\Psi(\mathscr{P}_{n}(\mathbb{C}))$ of $\mathscr{P}_{2n}(\mathbb{R})$\,.

\noindent Let $\sigma_{1}$ be a Kubo-Ando mean on $\mathscr{P}_{n}(\mathbb{C})$. The mean $\sigma_{1}$ defines a map $\widetilde{\sigma}_{1}$:
\begin{equation*}
\widetilde{\sigma}_{1}: \widetilde{\mathscr{P}}_{n}(\mathbb{C}) \times \widetilde{\mathscr{P}}_{n}(\mathbb{C}) \mapsto \widetilde{\mathscr{P}}_{n}(\mathbb{C})
\end{equation*}
given by
\begin{equation}
\A \widetilde{\sigma}_{1} \B = \Psi\left(\Psi^{-1}(\A) \sigma_{1} \Psi^{-1}(\B)\right)\,, \qquad \left(\A\,, \B \in \widetilde{\mathscr{P}}_{n}(\mathbb{C})\right)\,.
\label{CompatibilityPsiMean}
\end{equation}

\noindent As explained in \cite{FRANCOMERINO2}, we have a bijective correspondence between Kubo-Ando means on $\mathscr{P}_{n}(\mathbb{C})$ and $\mathscr{P}_{2n}(\mathbb{R})$. More precisely, there exists a unique Kubo-Ando mean $\sigma_{2}$ of $\mathscr{P}_{2n}(\mathbb{R})$ such that the restriction of $\sigma_{2}$ to $\widetilde{\mathscr{P}}_{n}(\mathbb{C})$ is $\widetilde{\sigma}_{1}$, and such that
\begin{equation*}
\Psi(\A \sigma_{1} \B) = \Psi(\A) \sigma_{2} \Psi(\B)\,, \qquad \left(\A\,, \B \in \mathscr{P}_{n}(\mathbb{C})\right)\,.
\end{equation*}

\begin{theo}

Suppose that $\A\,, \B \in \mathscr{P}_{n}(\mathbb{C})$ such that $\left|\Spec(\A^{-1}\B)\right| = r$, and that
\begin{equation*}
\A \sigma_{1} \B = \sum\limits_{i=0}^{r-1} c_{i}\,\A(\A^{-1}\B)^{i}\,,
\end{equation*}
where the coefficients $c_{0}\,, \ldots\,, c_{r-1}$ are given by Theorem~\ref{TheoremGeneralDecomposition}. Then
\begin{equation}
\Psi(\A) \sigma_{2} \Psi(\B) = \sum\limits_{i=0}^{r-1} c_{i} \Psi(\A)\bigl(\Psi(\A)^{-1}\Psi(\B)\bigr)^{i}\,.
\label{EquationPsiAPsiB}
\end{equation}
In particular, the decomposition of Theorem~\ref{TheoremGeneralDecomposition} is preserved under the embedding $\Psi$ and the coefficients $c_{i}$ are unchanged\,.

\label{TheoremEmbedding}

\end{theo}

\begin{proof}

Let $\A\,, \B \in \mathscr{P}_{n}(\mathbb{C})$ such that $\left|\Spec(\A^{-1}\B)\right| = r$ and 
\begin{equation*}
\A \sigma_{1} \B = \sum\limits_{i=0}^{r-1} c_{i}\,\A(\A^{-1}\B)^{i}\,,
\end{equation*}
Using that the restriction of $\sigma_{2}$ to $\widetilde{\mathscr{P}}_{n}(\mathbb{C})$ is equal to $\widetilde{\sigma}_{1}$, it follows from Equation \eqref{CompatibilityPsiMean} that
\begin{eqnarray*}
\Psi(\A) \sigma_{2} \Psi(\B) & = & \Psi(\A) \widetilde{\sigma}_{1} \Psi(\B) = \Psi\left(\Psi^{-1}(\Psi(\A)) \sigma_{1} \Psi^{-1}(\Psi(\B))\right) \\ 
& = & \Psi(\A \sigma_{1} \B) = \Psi\left(\sum\limits_{i=0}^{r-1} c_{i}\A(\A^{-1}\B)^{i}\right) = \sum\limits_{i = 0}^{r-1} c_{i} \Psi\left(\A(\A^{-1}\B)^{i}\right) \\ 
& = & \sum\limits_{i = 0}^{r-1} c_{i} \Psi(\A)\Psi\left((\A^{-1}\B)^{i}\right) = \sum\limits_{i = 0}^{r-1} c_{i} \Psi(\A)\left(\Psi(\A)^{-1}\Psi(\B))^{i}\right)\,.
\end{eqnarray*}
Moreover, as explained in \cite{FRANCOMERINO2},
\begin{equation*}
\Spec\left(\Psi(\A)^{-1}\Psi(\B)\right) = \Spec\left(\Psi(\A^{-1}\B)\right) = \Spec(\A^{-1}\B)\,,
\end{equation*}
the only difference being that each eigenvalue occurs with twice its multiplicity. Hence
\begin{equation*}
\left|\Spec\left(\Psi(\A)^{-1}\Psi(\B)\right)\right| = \left|\Spec(\A^{-1}\B)\right| = r\,.
\end{equation*}
i.e. the decomposition given in \ref{EquationPsiAPsiB} is minimal, so the degree of the interpolating polynomial cannot be reduced\,.

\end{proof}

\begin{remark}

In \cite{FRANCOMERINO2}, we also defined an embedding of the cone $\mathscr{P}_{n}(\mathbb{H})$ into $\mathscr{P}_{2n}(\mathbb{C})$ and proved, as explained above, a one-to-one correspondence between the Kubo-Ando means on $\mathscr{P}_{n}(\mathbb{H})$ and $\mathscr{P}_{2n}(\mathbb{C})$. One can show that Theorem \ref{TheoremEmbedding} is still valid if $\A$ and $\B$ are two matrices in $\mathscr{P}_{n}(\mathbb{H})$, i.e. the polynomial coefficients of $\A \sigma \B$ and $\Psi(\A) \sigma \Psi(\B)$ are unchanged\,.

\end{remark}

\section{Extension to alternative means}

We start with a definition \cite{DumitruFrancoKimCzerwinska2026}.

\begin{definition}

Let $f: \left(0\,, \infty\right) \rightarrow \left(0\,, \infty\right)$ be a continuous function such that $f(1)=1$. The \emph{alternative mean associated with $f$} is the binary operation
\begin{equation*}
\hat{\sigma}_{f}: \mathscr{P}_{n}(\mathbb{D}) \times \mathscr{P}_{n}(\mathbb{D}) \to \mathscr{P}_{n}(\mathbb{D})
\end{equation*}
defined by
\begin{equation*}
\A \hat{\sigma}_{f} \B = f\left(\A^{-1} \sharp \B\right)\A f\left(\A^{-1}\sharp \B\right)\,, \qquad \left(\A\,, \B \in \mathscr{P}_{n}(\mathbb{D})\right)\,,
\end{equation*}
where $\sharp$ denotes the geometric mean on $\mathscr{P}_{n}(\mathbb{D})$.

\end{definition}

\begin{remark}

Since $f(1)=1$, we have
\begin{equation*}
\A \hat{\sigma}_{f} \A = f(\Id_{n})\A f(\Id_{n}) = \A\,,
\end{equation*}
for every $\A \in \mathscr{P}_{n}(\mathbb{D})$. Thus, $\hat{\sigma}_{f}$ is normalized in the sense that the mean of a matrix with itself is equal to the matrix\,.

\end{remark}

\noindent In this section, we take $\A\,, \B \in \mathscr{P}_{n}(\mathbb{D})$ such that $\A\B = \B\A$. In this case, we have
\begin{equation*}
\A^{-1} \sharp \B = \left(\A^{-1}\B\right)^{\frac{1}{2}}\,.
\end{equation*}
Indeed, when $\A\B = \B\A$ the matrices $\A$ and $\B$ are simultaneously diagonalizable, and the identity follows directly from the definition of the geometric mean\,.

\begin{remark}

From the equation $\A\B = \B\A$, we get that
\begin{equation*}
\left(\A^{-1}\B\right)\A = \A\left(\A^{-1}\B\right)\,, \qquad \left(\A^{-1}\B\right)\B = \B\left(\A^{-1}\B\right)\,.
\end{equation*}
More generally, for all $k \in \mathbb{N}$, we have
\begin{equation*}
\left(\A^{-1}\B\right)^{\frac{k}{2}}\A = \A\left(\A^{-1}\B\right)^{\frac{k}{2}}\,, \qquad \left(\A^{-1}\B\right)^{\frac{k}{2}}\B = \B\left(\A^{-1}\B\right)^{\frac{k}{2}}\,,
\end{equation*}
and
\begin{equation*}
f(\A^{-1} \sharp \B)\A = \A f(\A^{-1} \sharp \B)\,, \qquad f(\A^{-1} \sharp \B)\B = \B f(\A^{-1} \sharp \B)\,.
\end{equation*}
\label{RemarkAlternativeMeans}

\end{remark}

\noindent Therefore, it follows from Remark \ref{RemarkAlternativeMeans} that
\begin{equation}
\A \hat{\sigma}_{f} \B = \A f\left(\A^{-1} \sharp \B\right)^{2} = \A f\left((\A^{-1}\B)^{\frac{1}{2}}\right)^{2}\,.
\label{EquationAlternativeMeans}
\end{equation}

\noindent Let $r = \left|\Spec(\A^{-1}\B)\right|$, with $\Spec(\A^{-1}\B) = \left\{\lambda_{1}\,, \ldots\,, \lambda_{r}\right\}$. Then $\Spec\left(\left(\A^{-1}\B\right)^{\frac{1}{2}}\right) = \left\{\sqrt{\lambda_{1}}\,, \ldots\,, \sqrt{\lambda_{r}}\right\}$. Therefore
\begin{equation*}
r = \left|\Spec(\A^{-1}\B)\right| = \left|\Spec\left(\left(\A^{-1}\B\right)^{\frac{1}{2}}\right)\right|\,.
\end{equation*}

\noindent As explained in Section \ref{SectionDecompositionKAMeans}, there exists a unique polynomial $\Q$ of degree $r-1$ such that
\begin{equation*}
\Q(\sqrt{\lambda_{i}}) = f(\sqrt{\lambda_{i}})\,, \qquad \left(1 \leq i \leq r\right)\,,
\end{equation*}
and by functional calculus, we have
\begin{equation*}
f\left(\left(\A^{-1}\B\right)^{\frac{1}{2}}\right) = \Q\left(\left(\A^{-1}\B\right)^{\frac{1}{2}}\right)\,.
\end{equation*}
More precisely, using Equation \eqref{CoefficientsCK}, the polynomial $\Q$ is of the form 
\begin{equation}
\Q(t) = \sum\limits_{i = 0}^{r-1} d_{i}t^{i}\,,
\label{PolynomialQ}
\end{equation}
where the constant $d_{k}$ are given by
\begin{equation}
d_{k} = (-1)^{r-1-k}\sum\limits_{i = 1}^{r} f(\sqrt{\lambda_{i}})\frac{e_{r-1-k}(\widehat{\sqrt{\lambda_{i}}})}{\prod\limits_{j \neq i}\left(\sqrt{\lambda_{i}} - \sqrt{\lambda_{j}}\right)}\,, \qquad \left(0 \leq k \leq r-1\right)\,.
\end{equation}

\begin{theo}

Let $\A\,, \B \in \mathscr{P}_{n}(\mathbb{D})$ be such that $\A\B = \B\A$, and let $r = \left|\Spec(\A^{-1}\B)\right|$. Then
\begin{equation*}
f\left((\A^{-1}\B)^{\frac{1}{2}}\right) = \sum\limits_{i=0}^{r-1} d_{i}(\A^{-1}\B)^{\frac{i}{2}}\,,
\end{equation*}
where the coefficients $d_{0}\,, \ldots\,, d_{r-1}$ are given above.

\noindent Moreover,
\begin{equation*}
\A \hat{\sigma}_{f} \B = \sum\limits_{i=0}^{r-1} \sum\limits_{j=0}^{r-1} d_{i}d_{j}\A(\A^{-1}\B)^{\frac{i+j}{2}}\,.
\end{equation*}

\label{TheoremAlternativeMeans}

\end{theo}

\begin{proof}

Using that $f\left(\left(\A^{-1}\B\right)^{\frac{1}{2}}\right) = \Q\left(\left(\A^{-1}\B\right)^{\frac{1}{2}}\right)$, it follows from Equation \eqref{PolynomialQ} that
\begin{equation*}
f\left((\A^{-1}\B)^{\frac{1}{2}}\right) = \sum\limits_{i=0}^{r-1} d_{i}(\A^{-1}\B)^{\frac{i}{2}}\,.
\end{equation*}
Therefore, it follows from Equation \eqref{EquationAlternativeMeans} that
\begin{equation*}
\A \hat{\sigma}_{f} \B = \A f\left((\A^{-1}\B)^{\frac{1}{2}}\right)^{2} = \A\left(\sum\limits_{i=0}^{r-1} d_{i}(\A^{-1}\B)^{\frac{i}{2}}\right)^{2} = \sum\limits_{i=0}^{r-1} \sum\limits_{j=0}^{r-1} d_{i}d_{j}\A(\A^{-1}\B)^{\frac{i+j}{2}}\,.
\end{equation*}

\end{proof}

\begin{remark}

Assume that $r = 2$. Then Theorem \ref{TheoremAlternativeMeans} gives
\begin{equation*}
\A \hat{\sigma}_{f} \B = \sum\limits_{i=0}^{1}\sum\limits_{j=0}^{1} d_{i}d_{j}\A\left(\A^{-1}\B\right)^{\frac{i+j}{2}}\,,
\end{equation*}
that is,
\begin{equation*}
\A \hat{\sigma}_{f} \B = d^{2}_{0}\A + 2d_{0}d_{1}\A\left(\A^{-1}\B\right)^{\frac{1}{2}} + d^{2}_{1}\B\,.
\end{equation*}
Since $\left|\Spec(\A^{-1}\B)\right|=2$, Theorem \ref{TheoremGeneralDecomposition} applied to the function $t \to t^{\frac{1}{2}}$ implies that there exist constants $\alpha\,, \beta \in \mathbb{R}$ such that
\begin{equation*}
\left(\A^{-1}\B\right)^{\frac{1}{2}} = \alpha\Id_{n}+\beta\A^{-1}\B\,.
\end{equation*}
Multiplying the previous equation by $\A$ on the left-hand side, we obtain
\begin{equation*}
\A\left(\A^{-1}\B\right)^{\frac{1}{2}} = \alpha\A+\beta\B\,.
\end{equation*}
Substituting into the previous expression yields
\begin{equation*}
\A \hat{\sigma}_{f} \B = \left(d^{2}_{0} + 2\alpha d_{0}d_{1}\right)\A + \left(d^{2}_{1} + 2\beta d_{0}d_{1}\right)\B.\,
\end{equation*}
Hence every alternative mean admits a linearization of the form
\begin{equation*}
\A \hat{\sigma}_{f} \B = \gamma\A+\delta\B
\end{equation*}
whenever $\A\B=\B\A$ and $\left|\Spec(\A^{-1}\B)\right|=2$. In particular, we recover and extend the linearization phenomenon established by Choi, Kim, and Lim for the spectral geometric mean to
the whole class of alternative means in \cite{CHOIKIMLIM}\,.

\noindent More generally, it follows from Proposition \ref{PropositionExsitenceP} that for all $0 \leq p \leq r-2$, there exists a polynomial $\P_{p}$ or degree at most $r-1$ such that
\begin{equation*}
\left(\A^{-1}\B\right)^{\frac{2p+1}{2}} = \P_{p}(\A^{-1}\B)\,.
\end{equation*}
Therefore, we get from Theorem \ref{TheoremAlternativeMeans} that there exists $\alpha_{0}\,, \alpha_{1}\,, \ldots\,, \alpha_{r-1} \in \mathbb{R}$ such that
\begin{equation*}
\A \hat{\sigma}_{f} \B = \sum\limits_{k = 0}^{r-1} \alpha_{k}\A\left(\A^{-1}\B\right)^{k}\,,
\end{equation*}
and where the coefficients $\alpha_{k}\,, 0 \leq k \leq r-1$ can be computed explicitly from the coefficients $d_{j}$ and the interpolation polynomials $\P_{p}$. Consequently, every alternative mean admits a decomposition of the same form as in Theorem \ref{TheoremGeneralDecomposition}\,.

\label{LastRemark}

\end{remark}

\begin{notation}

Let $g: \left(0\,, \infty\right) \to \left(0\,, \infty\right)$ be the function given by
\begin{equation*}
g(t) = f(\sqrt{t})^{2}\,, \qquad \left(t \in \left(0\,, \infty\right)\right)\,.
\end{equation*}

\end{notation}

\noindent We now extend the results of \cite[Theorem~5.1]{CHOIKIMLIM} in the case where $g$ is strictly convex or stricly concave\,.

\begin{theo}

Suppose that $\A\B = \B\A$ and that $g$ is strictly convex or stricly concave. Then $\A \hat{\sigma}_{f} \B$ is linearizable if and only if $\left|\Spec(\A^{-1}\B)\right| \leq 2$\,.

\label{ExtensionCHOIKIMLIM}

\end{theo}

\begin{proof}

Assume that there exists $\alpha\,, \beta \in \mathbb{R}$ such that
\begin{equation}
\A \hat{\sigma}_{f} \B = \alpha\A + \beta\B\,.
\label{LinearizationHatSigma}
\end{equation}
Let $\X = \A^{-1}\B$. From the previous discussion, we get 
\begin{equation*}
g(\X) = f\left(\sqrt{\X}\right)^{2}\,,
\end{equation*}
so it follows from Equation \eqref{EquationAlternativeMeans} that 
\begin{equation}
\A \hat{\sigma}_{f} \B = \A f\left(\X^{\frac{1}{2}}\right)^{2} = \A g(\X)\,.
\label{LastEquation}
\end{equation}
Therefore, by combining Equations \eqref{LinearizationHatSigma} and \eqref{LastEquation}, it follows that
\begin{equation*}
g(\X) = \alpha\Id_{n} + \beta\X\,.
\end{equation*}
Let $\lambda \in \Spec(\X)$. Then
\begin{equation*}
g(\lambda) = \alpha + b\lambda\,.
\end{equation*}
Since $g$ is strictly convex or strictly concave, the graph of $g$ cannot intersect the line $t \mapsto \alpha+\beta t$ at more than two points. Hence
\begin{equation*}
\left|\Spec(\X)\right| \leq 2\,.
\end{equation*}
The converse follows from Remark \ref{LastRemark}

\end{proof}

\begin{remark}

For the function
\begin{equation*}
f(t) = \frac{1+t}{2}\,,
\end{equation*}
we have
\begin{equation*}
g(t) = f(\sqrt{t})^{2} = \left(\frac{1+\sqrt{t}}{2}\right)^{2} = \frac{1}{4}+\frac{1}{2}\sqrt{t}+\frac{t}{4}\,.
\end{equation*}
Since
\begin{equation*}
g''(t) = -\frac{1}{8t^{3/2}}\,, \qquad \left(t > 0\right)\,,
\end{equation*}
the function $g$ is strictly concave on $\left(0\,, \infty)\right)$. Hence the previous theorem recovers the linearization result of Choi, Kim, and Lim for the Wasserstein mean \cite[Section~5]{CHOIKIMLIM}\,.

\noindent The theorem also applies to many other alternative means. 
\begin{itemize}
\item If $f(t) = t^{\alpha}$, with $\alpha > 0$, i.e.
\begin{equation*}
g(t) = f(\sqrt{t})^{2} = t^{\alpha}\,.
\end{equation*}
Then, whenever $\alpha \neq 1$, the function $g$ is strictly convex or strictly concave on $\left(0\,, \infty)\right)$, and the same characterization holds\,.
\item If $f(t)=\frac{2t}{1+t}$, we obtain
\begin{equation*}
g(t) = \frac{4t}{(1+\sqrt{t})^{2}}\,,
\end{equation*}
which gives another nontrivial example covered by Theorem \ref{ExtensionCHOIKIMLIM}\,.
\item Finally, if $f(t) = (1-\alpha) + \alpha t$, with $\alpha \in \left(0\,, 1\right)$. Then, the corresponding $g$ is given by the function generating the quasi-Wasserstein means with weight $\alpha$ and power $\frac{1}{2}$
\begin{equation*}
g(t) = \left((1-\alpha) + \alpha\sqrt{t}\right)^{2}\,.
\end{equation*}
The function $g$ is operator monotone, hence operator concave (see \cite{DFK})\,.
\end{itemize}
\end{remark}

In this section, we assumed that $\A$ and $\B$ commute. We can still obtain a decomposition of $\A \hat{\sigma}_{f} \B$ even in the non-commutative setting. Indeed, let $\left(\A\,, \B\right)$ be any pair in $\mathscr{P}_{n}(\mathbb{D})^{2}$. Let $r = \Spec(\A^{-1} \sharp \B)$. Using Proposition \ref{PropositionExsitenceP}, there exists a polynomial $\P(t) = \sum\limits_{i = 0}^{r-1} e_{i}t^{i}$ such that $f(\A^{-1} \sharp \B) = \P(\A^{-1} \sharp \B)$, where the coefficients $e_{i}$ depends on $f$ and the eigenvalues of $\A^{-1} \sharp \B$ (see Equation \eqref{CoefficientsCK}). Then
\begin{equation*}
\A \hat{\sigma}_{f} \B = \sum\limits_{i = 0}^{r-1}\sum\limits_{j = 0}^{r-1} e_{i}e_{j}\left(\A^{-1} \sharp \B\right)^{i}\A\left(\A^{-1} \sharp \B\right)^{j}\,.
\end{equation*}
In contrast with the commuting case, the main difficulty is that the powers
$\left(\A^{-1}\sharp\B\right)^{k}$ do not generally admit a simple expression
in terms of $\A$ and $\B$\,.

\end{document}